\documentclass{article}

\usepackage[utf8]{inputenc}
\usepackage[T1]{fontenc}

\usepackage{fullpage}
\usepackage{abstract}
\usepackage{indentfirst}
\usepackage{enumerate}
\usepackage{enumitem}
\usepackage[shortcuts]{extdash}	

\usepackage{amsmath}
\usepackage{amsthm}
\usepackage{amsfonts}
\usepackage{MnSymbol}
\usepackage[mathcal]{euscript}			

\usepackage{tikz-cd}

\usepackage{hyperref}
\usepackage{xcolor}
\hypersetup{
    colorlinks=true,
    linkcolor=[RGB]{13 71 161},		
    citecolor=[RGB]{13 71 161},		
    urlcolor=[RGB]{13 71 161}		
}
\usepackage[capitalize,nameinlink]{cleveref}

\newtheorem{thm}{Theorem}[section]
\newtheorem{theorem}[thm]{Theorem}
\newtheorem{proposition}[thm]{Proposition}
\newtheorem{corollary}[thm]{Corollary}
\newtheorem{lemma}[thm]{Lemma}
\theoremstyle{definition}
\newtheorem{definition}[thm]{Definition}
\newtheorem{example}[thm]{Example}
\newtheorem{examples}[thm]{Examples}
\newtheorem{remark}[thm]{Remark}

\newtheorem{warning}[thm]{Warning}
\newtheorem{question}[thm]{Question}

\newlist{exmpenum}{enumerate}{1}
\setlist[exmpenum]{label=\alph*), ref=\theproposition~(\alph*)}
\crefalias{exmpenumi}{example}

\newcommand{\NN}{\mathbb{N}}
\newcommand{\RR}{\mathbb{R}}
\newcommand{\ZZ}{\mathbb{Z}}

\newcommand{\cA}{\mathcal{A}}

\newcommand{\cC}{\mathcal{C}}
\newcommand{\cD}{\mathcal{D}}
\newcommand{\cE}{\mathcal{E}}
\newcommand{\cF}{\mathcal{F}}

\newcommand{\cP}{\mathcal{P}}

\newcommand{\cR}{\mathcal{R}}
\newcommand{\cS}{\mathcal{S}}

\newcommand{\cU}{\mathcal{U}}

\newcommand\ot\leftarrow
\newcommand\op{^{op}}
\newcommand\oo{$\infty$\=/}
\newcommand\ooo{$(\infty,1)$\=/}
\newcommand\truncated[1]{^{\leq #1}}

\newcommand\slice[1]{_{/#1}}	

\newcommand\surj{\twoheadrightarrow}
\newcommand\Ra{\Rightarrow}
\renewcommand\iff{\Leftrightarrow}
\renewcommand\index{_\bullet}
\newcommand\tower{\truncated\bullet}

\newcommand\hred{\truncated\infty}
\newcommand\infconn{^{(\infty)}}

\DeclareMathOperator*{\colim}{colim}

\renewcommand{\S}[1]{{\cS[#1]}}	
\newcommand\Fin{\mathrm{\cF in}}
\newcommand\PSp{\mathrm{\cP\cS p}}

\newcommand\pbmark{\ar[dr, phantom, "\ulcorner" very near start, shift right=1ex]}
\newcommand\pbmarkk{\ar[drr, phantom, "\ulcorner" very near start, shift right=1ex]}

\newcommand\Set{\mathrm{\cS et}}

\newcommand\Fun[2]{\mathrm{Fun}\left(#1,#2\right)}

\newcommand\join\star

\newcommand\Tower{\overleftarrow\NN}

\DeclareMathOperator{\Map}{Map}

\newcommand\Post[1]{\mathrm{\cP ost}\left(#1\right)}

\newcommand\CC[1]{\mathsf{CC}_{#1}}
\newcommand\DC[1]{\mathsf{DC}_{#1}}

\renewcommand{\S}[1]{{\cS[#1]}}	

\newcommand{\isConn}{\mathsf{isConn}}

\definecolor{mat}{HTML}{FFD6AD}

\definecolor{reid}{HTML}{ADFFD6}

\title{
Choice axioms and Postnikov completeness
\footnote{This material is based upon work supported by the Air Force Office of Scientific Research under award numbers FA9550-20-1-0305 and FA9550-21-1-0009.}
}

\author{
Mathieu Anel
\footnote{
Department of Philosophy, 
Carnegie Mellon University,
\href{mailto:mathieu.anel@protonmail.com}{mathieu.anel@protonmail.com}
}
\and
Reid Barton
\footnote{
Department of Philosophy, 
Carnegie Mellon University,
\href{mailto:barton2@andrew.cmu.edu}{barton2@andrew.cmu.edu}
}
}

\begin{document}

\maketitle

\begin{abstract}
We introduce homotopical variants of the axioms of countable and dependent choice for \oo topoi and use them to give criteria for Postnikov completeness, revisiting a result of Mondal and Reinecke.
\end{abstract}

\setcounter{tocdepth}{1}
\tableofcontents

\section{Introduction}

The axiom of countable choice in a 1\=/topos says that
a countable product of surjections is a surjection.
An example of a 1\=/ topos that does \emph{not} satisfy countable choice
is the topos of sheaves of sets on $[0,1] \subseteq \RR$.
However, by considering the \oo topos of sheaves of spaces on $[0,1]$, we can formulate a weaker condition that does hold:
a countable product of \emph{connected} maps is a surjection.

More generally, we say that \emph{countable choice of dimension $\leq d$} ($\CC d$) holds in an $n$\=/topos $\cE$ ($1\leq n\leq\infty$) if countable products of $(d-1)$\=/connected maps are surjections in $\cE$.
We always have $\CC d \Ra \CC {d+1}$;
the strongest axiom $\CC 0$ is the usual axiom of countable choice.
For each $d \ge 1$, the \oo topos of sheaves on $[0,1]^d$
satisfies $\CC d$ but not $\CC {d-1}$ (\cref{ex:cube}).
It is also possible for $\CC d$ to fail for every $d$, even $d = \infty$ (\cref{prop:noCC}).

In connection with these axioms, we prove the following results.

\begin{theorem}
\label{thm:main}
If $\cE$ is an \oo topos where $\CC d$ holds for some $d<\infty$,
then the hypercompletion and the Postnikov completion of $\cE$ coincide.
\end{theorem}

\begin{theorem}
\label{thm:CC:1-to-infty}
If $\cE_1$ is a 1-topos where $\CC 0$ holds,
then $\CC 0$ also holds in its \oo topos envelope $\cE$.
\end{theorem}

\Cref{thm:main} says concretely that every formal Postnikov tower in $\cE$ is the Postnikov tower of its limit (see \cref{app:lem:Postnikov}).
\Cref{thm:CC:1-to-infty} implies that this holds in the \oo topos envelope of a 1-topos $\cE_1$ with $\CC 0$.
In \cite{Mondal-Reinecke:Postnikov},
answering a question of \cite{Bhatt-Scholze:proetale},
the authors prove this last fact under the stronger condition that $\cE_1$ is ``replete'', meaning that the composition of a countable tower of surjections is a surjection.
This property is also known as the axiom of \emph{dependent choice}.
We say that the axiom of \emph{dependent choice of dimension $\leq d$} ($\DC d$) holds in an $n$\=/topos if the composition of a countable tower of $(d-1)$\=/connected maps is a surjection (see \cref{def:DC}).
We always have $\DC d\Ra\DC {d+1}$ and $\DC d\Ra \CC d$.
The strongest axiom $\DC 0$ is the repleteness condition above.
The relation $\DC 0\Ra\CC 0$ lets us reprove \cite[Theorem A]{Mondal-Reinecke:Postnikov} in \cref{cor:mondal-reinecke}.

The axioms $\CC{}$ and $\DC{}$ have a tight relationship with the notion of \emph{homotopy dimension} (\cref{def:hd}).
Informally, an ``operation'' in a topos is of homotopy dimension $\leq d$ if it sends $(m+d)$\=/connected maps into $m$\=/connected maps for all $m$.
The axiom $\CC d$ says that countable products are of homotopy dimension $\leq d$ (\cref{def:CC}).
We shall see that $\DC d$ holds in any $n$\=/topos with enough objects of homotopy dimension $\leq d$ (see \cref{def:enough,prop:enough-DCd}).

As an application of \cref{thm:main}, we give in \cref{cor:lhd-PC} a new proof that an \oo topos locally of homotopy dimension $\leq d$ is Postnikov complete \cite[Proposition 7.2.1.10]{Lurie:HTT}.
We also show that \cref{thm:main} is false for $d=\infty$ (\cref{rem:main:infty}), and that the $\CC d$ condition of \cref{thm:main} is not necessary for the hypercompletion and Postnikov completion to agree (\cref{rem:main:infty:2}).

\medskip
In an appendix, we summarize the various convergence properties for Postnikov towers and introduce some new terminology (\cref{app:def:Postnikov}).

\medskip

\paragraph{Acknowledgments.} 
We are thankful to Mike Shulman for his comments on an earlier draft of the paper.

\paragraph{Conventions.}
We set our conventions about categories, $n$\=/topoi and connectivity.
We shall say simply \emph{category} to refer to an \ooo category.
By default, a category is not assumed small.
An \emph{$n$\=/category} is a category whose mapping spaces are all $(n-1)$\=/truncated.
We denote the category of (small) spaces by $\cS$.
For $-1\leq n\leq \infty$, we denote by $\cS\truncated n$ the $(n+1)$\=/category of $n$\=/truncated spaces.
For $n=0$, $\cS\truncated 0=\Set$ is the 1-category of sets, and for $n=\infty$, we put $\cS\truncated \infty=\cS$ by convention.
An $n$\=/category $\cE$ is an \emph{$n$\=/topos} if it is an (accessible) left-exact localization of $\Fun C {\cS\truncated {n-1}}$ for a small $n$\=/category $C$.
(When $n<\infty$, all left-exact localizations of $\Fun C {\cS\truncated {n-1}}$ are accessible \cite[Remark 6.4.1.2]{Lurie:HTT}. This is open for $n=\infty$.)
An \emph{algebraic morphism} (of $n$\=/topoi) $u^*:\cE\to \cF$ is a cocontinuous and left-exact functor.
Such a functor always has a right adjoint $u_*$.

We shall say that a map in an $n$\=/topos is a \emph{surjection}, or that it is \emph{surjective}, if it is an effective epimorphism in the sense of \cite[Section 6.2.3]{Lurie:HTT}.
Recall that for a map $f:A\to B$, its diagonal is the map $\Delta(f):A\to A\times_BA$.
We put $\Delta^0(f) = f$ and $\Delta^{\ell+1}(f) = \Delta(\Delta^\ell(f))$.
For $-2\leq m\leq\infty$, we shall say that a map $f$ is \emph{$m$\=/connected} if, for every $0 \le \ell\leq m+1$, the map $\Delta^\ell(f)$ is a surjection.
Every map is $(-2)$\=/connected, and the $(-1)$\=/connected maps are the surjections.
A map is $m$\=/connected if and only if it is $(m+1)$\=/connective in the sense of \cite[Definition 6.5.1.10]{Lurie:HTT}.
An object $X$ of an $n$\=/topos is called \emph{$m$\=/connected} if $X\to 1$ is $m$\=/connected.
In the case $m=-1$, we shall also say that $X$ is \emph{inhabited}.

A map $f$ is \emph{$m$\=/truncated} if the diagonal $\Delta^{m+2}(f)$ is invertible.
The $(-2)$\=/truncated maps are the isomorphisms, and the $(-1)$\=/truncated maps are the monomorphisms.
For simplicity, we shall say that a map $f$ is \emph{\oo truncated} (rather than \emph{hypercomplete}) if it is right orthogonal to \oo connected maps.
For each $-2\leq m\leq\infty$, the classes of $m$\=/connected maps and $m$\=/truncated maps form a factorization system which is stable under base change (see \cite[Example 5.2.8.16 and Remark 6.5.2.21]{Lurie:HTT}, \cite[Proposition 3.3.6]{ABFJ:GBM}, or \cite[Lemma 2.2.22]{ABFJ:GT}).
We denote the $m$\=/truncation of an object $X$ by $X\truncated m$
(\cite{Lurie:HTT} uses the notation $\tau_{\le m} X$).

For $1\leq m\leq n\leq\infty$, the $m$\=/category of $(m-1)$\=/truncated objects in an $n$\=/topos is an $m$\=/topos \cite[Proposition 6.4.5.4]{Lurie:HTT}.
The \emph{$n$\=/topos envelope} of an $m$\=/topos is the left adjoint to this truncation operation \cite[Proposition 6.4.5.7]{Lurie:HTT}.
The original $m$\=/topos is the full subcategory of $(m-1)$\=/truncated objects of its $n$\=/topos envelope.
For $\cE$ an \oo topos, the \emph{hypercompletion} $P_\infty : \cE \to \cE\truncated\infty
$ is the localization inverting the class of \oo connected maps \cite[Section~6.5.2]{Lurie:HTT}.

\section{Choice axioms}

\subsection{Homotopy dimension}

We will formulate the axiom $\CC d$ in terms of the notion of homotopy dimension.

\begin{definition}[Homotopy dimension]
\label{def:hd}
Let $u^*:\cE\to \cF$ be an algebraic morphism of $n$\=/topoi and $u_*$ its right adjoint.
For $-1\leq d\leq \infty$, we shall say that $u^*$ is of \emph{homotopy dimension $\leq d$}
if for any $(d-1)$\=/connected map $A\to B$ in $\cF$, the map $u_*A\to u_*B$ is a surjection in $\cE$.
In the case that $\cE=\cS^{\le n-1}$ and $u^*:\cS^{\le n-1}\to \cF$ is the unique algebraic morphism of $n$\=/topoi, we shall say simply that $\cF$ is of homotopy dimension $\leq d$.
\end{definition}

The following lemma shows that the previous definition is equivalent to \cite[Definition 7.2.1.6]{Lurie:HTT}.

\begin{lemma}
\label{lem:decalage}
The morphism $u^* : \cE \to \cF$ is of homotopy dimension $\leq d$ if and only if, for every $m\geq -2$ and every $(m+d)$\=/connected map $A\to B$ in $\cF$, the map $u_*A\to u_*B$ is $m$\=/connected in $\cE$.
\end{lemma}
\begin{proof}
For the ``if'' direction, take $m=-1$.
For the converse, a map $f$ is $(m+d)$\=/connected if and only if $\Delta^\ell(f)$ is $(d-1)$\=/connected for every $0\leq \ell\leq m+1$.
Then $u_*(\Delta^\ell(f))$ is a surjection for $0\leq \ell\leq m+1$.
Using that $u_*\Delta(f) = \Delta(u_*f)$, we deduce that $u_*f$ is $m$\=/connected.
\end{proof}

\begin{examples}
\begin{exmpenum}
\item 
Recall from \cite[Theorem 7.2.3.6]{Lurie:HTT} that the \oo topos of sheaves on a topological space of covering dimension $\leq d$ is of homotopy dimension $\leq d$.
In particular, the \oo topos of sheaves on the cube $[0,1]^d$ is of 
homotopy dimension $\leq d$.

\item
In an $n$\=/topos for $n<\infty$, every $(n-1)$\=/connected map is invertible, and so all $n$\=/topoi and algebraic morphisms between $n$\=/topoi are of homotopy dimension $\leq n$.
(Beware the $m$-topos envelope of an $n$-topos, for $m>n$, may have a strictly larger homotopy dimension.)

\item
If an \oo topos is of homotopy dimension $\le \infty$, then for every \oo connected object $X$, the space of global sections of $X$ is $m$\=/connected for every $m$ by \cref{lem:decalage}, and therefore contractible.
This condition is in fact an equivalence by \cref{lem:hdim-le-iff} below.
We shall give examples of \oo topoi not of homotopy dimension $\le \infty$ in \cref{sect:noCC}.

\end{exmpenum}
\end{examples}

\begin{definition}
\label{def:external-dim}
Let $\cE$ be an $n$\=/topos.
We say that an object $X \in \cE$ is \emph{(externally) of homotopy dimension~$\le d$}
if the topos $\cE\slice X$ is of homotopy dimension $\le d$.
\end{definition}

Recall that a morphism of $\cE\slice X$ is $m$\=/connected
if and only if its underlying morphism in $\cE$ is
\cite[Proposition 6.5.1.19]{Lurie:HTT}.
The following lemma is essentially \cite[Lemma 7.2.1.7]{Lurie:HTT}.

\begin{lemma}
\label{lem:hdim-le-iff}
Let $\cE$ be an $n$\=/topos and $X$ an object of $\cE$.
The following are equivalent:
\begin{enumerate}
    \item\label{lem:hdim-le-iff:le}
    $X$ is of homotopy dimension $\le d$.
    \item\label{lem:hdim-le-iff:section}
    Every $(d-1)$\=/connected map $A \to X$ of $\cE$
    admits a section.
    \item\label{lem:hdim-le-iff:lift}
    For every $(d-1)$\=/connected map $A \to B$ of $\cE$,
    any map $X \to B$ admits a lift to $A$.
    \item\label{lem:hdim-le-iff:section-m}
    For every $m \ge -2$,
    the space of sections of any $(m+d)$\=/connected map of $\cE$
    is $m$\=/connected.
\end{enumerate}
\end{lemma}

\begin{proof}
    \eqref{lem:hdim-le-iff:le} $\Ra$ \eqref{lem:hdim-le-iff:section}:
    A $(d-1)$\=/connected map $f : A \to X$ of $\cE$
    amounts to a $(d-1)$\=/connected object of $\cE\slice X$,
    and a section of $f$ amounts to a global section of this object.
    Then the claim follows from the definition of homotopy dimension
    (with $B$ the terminal object of $\cE\slice X$).
    \eqref{lem:hdim-le-iff:section} $\Ra$ \eqref{lem:hdim-le-iff:lift}
    because the pullback of $A \to B$ along $X \to B$
    is again $(d-1)$\=/connected,
    and a section of this pullback is the same as
    a lift of the map $X \to B$ to $A$.
    \eqref{lem:hdim-le-iff:lift} $\Ra$ \eqref{lem:hdim-le-iff:le}:
    Let $A \to B$ be a $(d-1)$\=/connected map of $\cE\slice X$.
    By \eqref{lem:hdim-le-iff:lift}, any global section of $B$
    (i.e., section in $\cE$ of the structural map $B \to X$)
    lifts to some global section of $A$.
    Finally, \eqref{lem:hdim-le-iff:section} $\iff$ \eqref{lem:hdim-le-iff:section-m} by \cref{lem:decalage}.
\end{proof}

In particular, in a 1-topos, an object is of homotopy dimension $\le 0$
if and only if it is (externally) projective in the usual sense, as the $(-1)$\=/connected maps are the surjections.
(In \cite{Bhatt-Scholze:proetale}, projective objects are called ``weakly contractible''.)

\begin{remark}
\label{rmk:internal-hdim}
\Cref{def:external-dim} has a corresponding internal notion:
an object $X$ of an $n$\=/topos $\cE$ is
\emph{internally of homotopy dimension $\le d$}
if the algebraic morphism $X\times-:\cE \to \cE\slice X$
is of homotopy dimension $\le d$.
Informally, this means (using \cref{lem:decalage}) that
$X$\=/indexed dependent products in $\cE$
reduce connectivity by at most $d$.
When $d = 0$ and $n = 1$, the two conditions of being (internally) of homotopy dimension $\le d$
reduce to the conditions of being (internally) projective
(see \cite[Exercises IV.15--16]{MLM}).
\end{remark}

\medskip
In the proof of \cref{thm:main}, we will use the following auxiliary definition.
Let $K$ be an $n$\=/category (in particular, $K$ could be a set) and $\cE$ an $n$\=/topos.
The $K$\=/limit functor $\lim_K:\cE^K\to \cE$ is the right adjoint to an algebraic morphism $\Delta_K:\cE\to \cE^K$ (the constant diagram functor).

\begin{definition}[Homotopy $\cE$\=/dimension]
\label{def:hEd}
We shall say that $K$ is of \emph{homotopy $\cE$\=/dimension $\leq d$} if the algebraic morphism $\Delta_K:\cE\to \cE^K$ is of homotopy dimension $\leq d$.
\end{definition}

Below, a ``$K$\=/limit of maps'' refers to the limit of a $K$\=/indexed diagram in the arrow category.

\begin{lemma}
\label{lem:equiv-HD}
The following statements are equivalent:
\begin{enumerate}
\item\label{lem:equiv-HD:0} $K$ is of homotopy $\cE$\=/dimension $\leq d$.
\item\label{lem:equiv-HD:1} In $\cE$, $K$\=/limits of $(d-1)$\=/connected maps are surjections.
\item\label{lem:equiv-HD:2} In $\cE$, for every $m\geq -2$, $K$\=/limits of $(m+d)$\=/connected maps are $m$\=/connected maps.
\item\label{lem:equiv-HD:4} For every $E$ in $\cE$, $K$ is of homotopy $\cE\slice E$\=/dimension $\leq d$.
\item\label{lem:equiv-HD:5} There exists a surjection $E\surj 1$ such that $K$ is of homotopy $\cE\slice E$\=/dimension $\leq d$.
\item\label{lem:equiv-HD:1:obj} In every slice $\cE\slice E$, $K$\=/limits of $(d-1)$\=/connected objects are inhabited.
\item\label{lem:equiv-HD:2:obj} In every slice $\cE\slice E$, for every $m\geq -2$, $K$\=/limits of $(m+d)$\=/connected objects are $m$\=/connected in $\cE$.
\end{enumerate}
\end{lemma}
\begin{proof}
\eqref{lem:equiv-HD:0}$\iff$\eqref{lem:equiv-HD:1} by definition.
We show the remaining statements are equivalent to \eqref{lem:equiv-HD:1}.
\eqref{lem:equiv-HD:1}$\iff$\eqref{lem:equiv-HD:2} is a consequence of \cref{lem:decalage}.
\noindent \eqref{lem:equiv-HD:4}$\Ra$\eqref{lem:equiv-HD:1} is clear; let us see the converse.
Let $E$ be an object in $\cE$ and $A_k\to B_k\to E$ be a $K$\=/indexed diagram of $(d-1)$\=/connected maps in $\cE\slice E$.
We denote by $\lim\slice E A_k \to \lim\slice E B_k$ their limit in $\cE\slice E$.
This limit can be computed as a base change of the limit in $\cE$:
\[
\begin{tikzcd}
\lim\slice E A_k \ar[r]\ar[d] \pbmark & \lim\slice E B_k \ar[d]\ar[r] \pbmark & E \ar[d,"\Delta"]\\
\lim A_k \ar[r] & \lim B_k \ar[r] & \lim E \rlap{${} = E^{|K|}$}
\end{tikzcd}
\]
where $|K|$ is the groupoidification of $K$.
If $\lim A_k \to \lim B_k$ is a surjection, then so is $\lim\slice E A_k \to \lim\slice E B_k$ by base change.
This shows \eqref{lem:equiv-HD:1}$\Ra$\eqref{lem:equiv-HD:4}.
\eqref{lem:equiv-HD:1}$\Ra$\eqref{lem:equiv-HD:5} is clear,
and \eqref{lem:equiv-HD:5}$\Ra$\eqref{lem:equiv-HD:1} holds because
$E \times - : \cE \to \cE\slice E$ preserves limits and preserves and reflects connectivity of maps.
\eqref{lem:equiv-HD:4}$\Ra$\eqref{lem:equiv-HD:2:obj} follows from \cref{lem:decalage}, and \eqref{lem:equiv-HD:2:obj}$\Ra$\eqref{lem:equiv-HD:1:obj} is clear, so it remains to show \eqref{lem:equiv-HD:1:obj}$\Ra$\eqref{lem:equiv-HD:1}.
Let $A_k\to B_k$ be a $K$\=/indexed diagram of $(d-1)$\=/connected maps in $\cE$.
We put $A=\lim A_k$, $B=\lim B_k$ and $A'_k=A_k\times_{B_k}B$.
The maps $A'_k\to B$ are $(d-1)$\=/connected by base change.
By assumption on $\cE\slice B$, the object $\lim\slice B A'_k\to B$ is inhabited.
The conclusion follows from the remark that $A=\lim\slice B A'_k$.
\end{proof}

\medskip
The rest of this section contains results that will be useful to verify that a topos satisfies the choice axioms $\CC d$ or $\DC d$.

\begin{definition}[Enough objects]
\label{def:enough}
An $n$\=/topos has \emph{enough objects of homotopy dimension $\leq d$} if every object $X$ admits a covering family $(X_i\to X)$ where each $X_i$ is of homotopy dimension $\leq d$.
Recall from \cite[Definition 7.2.1.8]{Lurie:HTT} that an $n$\=/topos is \emph{locally of homotopy dimension $\leq d$} if it is generated under colimits by its subcategory of objects of homotopy dimension $\leq d$.
\end{definition}

\begin{example}
A 1\=/topos has enough objects of homotopy dimension $\leq 0$ if and only if it has enough (externally) projective objects in the usual sense.
\end{example}

\begin{remark}
\label{rem:enough}
If an $n$\=/topos is locally of homotopy dimension $\leq d$, then it has enough objects of homotopy dimension $\leq d$.
For $n<\infty$, the converse is also true.
For every object $X$, the existence of enough objects of homotopy dimension $\leq d$ lets us build a (semi-simplicial) hypercover $X\index$ of $X$ by objects of homotopy dimension $\leq d$.
By the analogue of \cite[Corollary A.5.3.3]{Lurie:SAG} for $n$\=/topoi, the colimit of this hypercover gives an isomorphism $|X\index|=X$, so the objects of homotopy dimension $\leq d$ generate the topos under colimits.

For $n=\infty$, however, the converse is false.
In the \oo topos of parametrized spectra $\PSp$, any set equipped with the zero spectrum is of homotopy dimension $\leq 0$.
Every object of $\PSp$ admits a surjection from such an object, so $\PSp$ has enough objects of homotopy dimension $\leq 0$.
This also means (by \cref{lem:hdim-le-iff}) that \emph{any} object of homotopy dimension $\leq 0$ must be a retract of such an object, hence itself a set equipped with the zero spectrum.
These objects only generate the subcategory $\cS\subseteq\PSp$ under colimits.
\end{remark}

\begin{warning}
An \oo topos of homotopy dimension $\le d$ need not be locally of homotopy dimension $\le d$.
By the previous remark, the \oo topos of parameterized spectra is a counterexample (with $d = 0$).
\end{warning}

\begin{lemma}
\label{lem:reflect-surj}
Let $\cE_n$ be an $n$\=/topos and $\cE$ its \oo topos envelope.
Then the inclusion functor $\cE_n\hookrightarrow \cE$ preserves and reflects surjections.
\end{lemma}
\begin{proof}
Recall that $\cE$ can be constructed as the category of higher sheaves on a small full subcategory $\cC \subseteq \cE_n$ for the canonical topology.
We may choose $\cC$ large enough to contain any particular surjection $f : A \to B$ of $\cE_n$.
(See the proofs of \cite[Propositions 6.4.3.6 and 6.4.5.7]{Lurie:HTT}.)
Then the image of $f$ in $\cE$ is also a surjection, by the definition of the canonical topology.
On the other hand, both in $\cE_n$ and in $\cE$, the surjections are characterized as maps that are the colimit of their \v Cech nerve \cite[Corollary 6.2.3.5]{Lurie:HTT}.
As $\cE_n \subseteq \cE$ is the full subcategory of $(n-1)$\=/connected objects, it is closed under finite limits and therefore this embedding commutes with the construction of the \v Cech nerve.
So if $f : A \to B$ in $\cE_n$ is such that $B$ is the colimit of the \v Cech nerve of $f$ in $\cE$, then this must also be the case in $\cE_n$.
\end{proof}

A similar argument also shows that every object $X$ in $\cE$ admits a surjection $Y\to X$ with $Y$ in $\cE_n$, because $\cE_n \subseteq \cE$ is closed under arbitrary coproducts.

\begin{lemma}
\label{lem:enough-0}
If an $n$\=/topos $\cE_n$ has enough objects of homotopy dimension $\leq 0$,
then the \oo topos envelope $\cE$ of $\cE_n$ also has enough objects of homotopy dimension $\leq 0$.
\end{lemma}
\begin{proof}
Choose a surjection $Y\to X$ in $\cE$ with $Y$ in $\cE_n$, and then a surjection $P\to Y$ in $\cE_n$, where $P$ is an object of homotopy dimension $\leq 0$.
The composition $P\to Y\to X$ is a surjection in $\cE$;
we have to verify that the object $P$ is still of homotopy dimension $\leq 0$ in $\cE$.
Let $f:Z\to P$ be a surjection in $\cE$. 
Choose a surjection $T\to Z$ with $T$ in $\cE_n$.
The composition $T\to Z\to P$ is a surjection in $\cE$ between objects of $\cE_n$.
By \cref{lem:reflect-surj}, it is a surjection in $\cE_n$.
Since $P$ is of homotopy dimension $\leq 0$ in $\cE_n$, the map $T\to P$ (and therefore the map $Z\to P$) admits a section.
This shows that $P$ is of homotopy dimension $\leq 0$ in $\cE$.
\end{proof}

\begin{remark}
\label{rem:enough-d}
For $0<d<n$, we do not know if the \oo topos envelope of an $n$\=/topos with enough objects of homotopy dimension $\leq d$ has enough objects of homotopy dimension $\leq d$ in general.
\end{remark}

\begin{lemma}
\label{lem:lochd-hc}
Any \oo topos $\cE$ locally of homotopy dimension $\leq d$ is hypercomplete.
\end{lemma}
\begin{proof}
Let $\cD\subseteq\cE$ be the full subcategory of objects of homotopy dimension $\leq d$.
By assumption, the nerve functor $N:\cE\to \Fun{\cD\op}\cS$ is conservative.
(If $f : X \to Y$ in $\cE$ is sent to an equivalence in $\Fun{\cD\op}\cS$,
then the class of objects $A \in \cE$ such that $f$ induces an equivalence $\Map(A,X) \to \Map(A,Y)$
contains $\cD$, and is closed under colimits.)
Since $\cD$ consists of objects of homotopy dimension $\leq d$,
$N$ sends $(m+d)$\=/connected maps to $m$\=/connected maps for all $m$.
In particular, it sends \oo connected maps to \oo connected maps, that is, to isomorphisms.
Since $N$ is conservative, every \oo connected map in $\cE$ is invertible.
\end{proof}

\begin{remark}
\label{rem:lochd-hc}
It is not necessarily the case that an \oo topos with enough objects of homotopy dimension $\leq d$ is hypercomplete. 
A counterexample for $d=0$ is given by the \oo topos of parametrized spectra.
\end{remark}

\begin{remark}
\label{rem:lochd-hc2}
It is not true that the \oo topos envelope of an $n$\=/topos locally of homotopy dimension $\leq d$ is itself locally of homotopy dimension $\leq d$, because this envelope need not be hypercomplete.
A counterexample for $d=0$ is given in \cite[Example A.10]{DHI}, see \cite[Example 3.29]{Mondal-Reinecke:Postnikov}.
\end{remark}

\subsection{Countable choice}

Given an $n$\=/topos $\cE$, we denote by $\NN$ the constant sheaf $\coprod_{n \in \NN} 1 \in \cE$ on the set of natural numbers.
By extensivity of sums in $\cE$, we have a canonical equivalence $\cE\slice \NN = \cE^\NN$, and the right adjoint to $\NN\times-:\cE\to \cE\slice \NN$ coincides with the product functor $\prod:\cE^\NN\to \cE$.

\begin{definition}[Countable choice]
\label{def:CC}
Let $\cE$ be an $n$\=/topos (for $1\leq n\leq \infty$). 
For $-1\leq d\leq \infty$, we say that the axiom of \emph{countable choice of dimension $\leq d$} ($\CC d$) holds in $\cE$
if $\NN$ is of homotopy $\cE$\=/dimension $\leq d$,
that is, if countable products of $(d-1)$\=/connected maps are surjections in $\cE$ (or if any of the equivalent characterizations of \cref{lem:equiv-HD} holds).
\end{definition}

\begin{examples}
\label{ex:CCd}
\begin{exmpenum}
\item\label{ex:CCd:strength} We always have $\CC {-1}\,\Ra\,\CC 0\,\Ra\,\CC 1\,\Ra\,\CC 2\,\Ra\,\cdots\,\Ra\,\CC\infty$.

\item\label{ex:CCd:slice} By \cref{lem:equiv-HD}~\eqref{lem:equiv-HD:4},  if $\CC d$ holds in $\cE$, then it holds in every slice of $\cE$.

\item\label{ex:CCd:d=-1} The axiom $\CC {-1}$ says that a countable product of arbitrary maps is a surjection. This holds if and only if $\cE=1$.

\item\label{ex:CCd:d=0} The axiom $\CC 0$ says that countable products of surjections are surjections. In a 1-topos, this is equivalent to the usual statement of countable choice expressed in the internal language \cite[Proposition~3.42]{Mejak:Choice}.
For $1\le n\le \infty$, $\CC 0$ holds in the $n$\=/topos $\cS\truncated {n-1}$ and in all diagram categories $\Fun C {\cS\truncated {n-1}}$ (in particular, in slices of $\cS\truncated {n-1}$).

\item\label{ex:CCd:d>n} For $d\geq n$, the axiom $\CC d$ holds trivially in every $n$\=/topos $\cE$, since the $(n-1)$\=/connected maps coincide with the isomorphisms.
In a 1-topos, then, the only non-trivial axiom is $\CC 0$.

\item\label{ex:CCd:infty} The axiom $\CC \infty$ says that countable products of \oo connected maps are surjections.
By \cref{lem:decalage} this is equivalent to saying that countable products of \oo connected maps are \oo connected maps. 
The axiom $\CC \infty$ holds for free in any hypercomplete \oo topos.
An example where it fails (and thus any $\CC d$ also) is given in \cref{prop:noCC}.

\item\label{ex:CCd:CCdnoCCd-1} The \oo topos of sheaves on the interval $[0,1]$ does not satisfy $\CC 0$.
For each $k \ge 1$, choose a cover $(U_{k,\alpha})$ of $[0,1]$ by open sets of diameter less than $1/k$,
and set $A_k = \coprod_\alpha U_{k,\alpha}$.
The product $\prod A_k$ is empty, since there is no nonempty open set $U \subseteq [0,1]$
small enough such that all of the objects $A_k$ have sections over $U$.
On the other hand, we will see below that this \oo topos does satisfy $\CC 1$.
More generally, for $0 \le d < \infty$, the \oo topos of sheaves on the $d$\=/cube $[0,1]^d$ satisfies $\CC d$ but not $\CC {d-1}$; see \cref{ex:cube}.

\item\label{ex:CCinfty:noCCd} The \oo topos of sheaves on the space $\coprod_d \, [0,1]^d$ satisfies $\CC \infty$, but does not satisfy $\CC d$ for any $d < \infty$.

\item We will see in \cref{prop:CC-for-hc} that if an \oo topos $\cE$ has $\CC d$, then so does the $n$-topos $\cE\truncated {n-1}$ for any $n$.

\item\label{ex:CCd:nohc}
The \oo topos of $\PSp$ of parametrized spectra satisfies $\CC 0$ without being hypercomplete.
Its hypercompletion is the functor $b:\PSp\to\cS$ extracting the base space.
The axiom $\CC 0$ holds in $\cS$ and it holds in $\PSp$ because $b$ preserves countable products and reflects surjections.

\end{exmpenum}	
\end{examples}

\begin{remark}
Let us explain intutively the geometric meaning of the axiom $\CC d$.
The failure of $\CC 0$ can be understood as follows.
Let $X$ be a topological space, $I$ a set, and $(A_i)$ a family of inhabited sheaves on $X$ (i.e. whose stalks are all nonempty).
A local section of $\prod A_i$ around a point $x$ is a neighborhood $U$ of $x$ and a section of every $A_i$ over $U$.
Every $A_i$ has local sections around $x$, but because the neighborhoods of $x$ are only closed under finite intersections, it might not be possible to find a section of all the $A_i$ over the same neighborhood of $x$.
For this reason, the stalk of $\prod A_i$ can be empty if the set $I$ is infinite, preventing the map $\prod A_i\to 1$ from being a surjection.
Of course, this does not happen if $x$ has enough neighborhoods where surjective sheaves always have sections (i.e. if $X$ is locally of homotopy dimension $\leq 0$ around $x$, like an Alexandrov space, or the Cantor space).

The axiom $\CC 1$ says that when the $A_i$ are 0-connected sheaves (i.e., gerbes), then the product $\prod A_i$ always has local sections around $x$.
However, the space of such sections might not be connected anymore.
This happens if $X$ is locally of covering dimension $\leq 1$ around $x$, for example in $X=[0,1]$.
The key remark is that given any neighborhood $U$ of $x$, we can always find a section of a sheaf $A$ on $U$ when $A$ is 0-connected.
Essentially, one starts with local sections on a cover $U_\alpha$ of $U$, and because $A$ is connected, one can find homotopies between these sections on the $U_{\alpha\beta}$, and use them to construct a global section on $U$.
Because $X$ is of covering dimension $\leq 1$, there are no triple intersections to consider.
(If $X$ had been of covering dimension $\leq 2$, we would have needed coherence homotopies on triple intersections. Such homotopies always exist if $A$ is a 1-connected sheaf.)
Given a family of connected sheaves $A_k$, we see that we can construct a section of each of $A_k$ on the same neighborhood of $x$.
Thus the stalk of $\prod A_i$ at $x$ is not empty.
\end{remark}

\begin{remark}
We now discuss the logical meaning of the axiom $\CC d$.
Recall first that, in ordinary logic, the class of surjections is used to provide semantics for \emph{existential conditions}.
Using the notation of \cite{hottbook}, a proposition $\exists x.P(x)$ is interpreted as the image of the map $p:\Sigma_{x:X} P(x)\to 1$, that is as the maximal subterminal object over which the map $p$ is a surjection.
The class of isomorphisms is similarly used for \emph{unique existential conditions}.
The proposition $\exists! x.P(x)$ is interpreted as the condition that the map $p:\Sigma_{x:X} P(x)\to 1$ is an isomorphism,
i.e., the maximal subterminal object over which the map $p$ is an isomorphism.
 
Between these two conditions, the classes of $n$-connected maps provide semantics for an (external) hierarchy of intermediate ``uniqueness levels'' for existential conditions.
Let us introduce the notation $\exists^n x.P(x)$ to mean that $\Sigma_{x:X} P(x)$ is an $n$-connected object.
Informally, for $n \ge 0$,
$\exists^n x.P(x)$ means that there exists $x$ satisfying $P(x)$ and for any two such $x$, there exists a path (or equality) between them in the sense of $\exists^{n-1}$.
Semantically, $\exists^n x.P(x)$ corresponds to the maximal subterminal object over which the map $p$ is $n$-connected.
When $n=-1$, this recovers the quantifier $\exists$.
When $n=\infty$, this does \emph{not} recover the quantifier $\exists!$, since the semantics of $\exists^\infty x.P(x)$ is that the object $\Sigma_{x:X} P(x)$ is \oo connected rather than contractible.
\[
\exists!
\ \Ra\ 
\exists^\infty
\ \Ra\ 
\dots
\ \Ra\ 
\exists^0
\ \Ra\ 
\exists^{-1}\,=\,\exists\,.
\]

With these notions in mind, any operation sending $(m+d)$-connected maps to $m$-connected maps (like $\prod_\NN$ in $\CC d$) can be thought of as an operation ``weakening the uniqueness level'' of existential conditions.

\end{remark}

\begin{remark}
\label{rem:int-dim-N}
Using the terminology of \cref{rmk:internal-hdim},
we can say that $\CC d$ holds in $\cE$
if and only if the object $\NN \in \cE$
is internally of homotopy dimension $\le d$.
This condition can also be expressed in the language of homotopy type theory, for instance as 
\[
\CC{}^\mathsf{int}:\NN \to \mathsf{Prop}\,,
\quad\qquad
\CC d^\mathsf{int} :=
\prod_{X : \NN \to \cU}
\big(\prod_{n : \NN} \isConn_{d-1}(X(n)) \big)
\to
\isConn_{-1}
\big( \prod_{n : \NN} X(n) \big)
\,.
\]
Here $\isConn_m(A)$ is the proposition that a type $A$ is $m$-connected,
i.e., that the $m$-truncation $\lVert A \rVert_m$ is contractible.
We set $\isConn_\infty(A) := \prod_{m : \NN} \isConn_m(A)$.
The statement $\CC d^\mathsf{int}$ is the appropriate formulation of $\CC d$ in an elementary \oo topos \cite{nlab:elementary,Rasekh:elementary}.
(In this context, the external statement of \cref{def:CC} may not even make sense since elementary topoi may not have external countable products when their natural number object is not standard.
For the same reason, the internal definition of $\isConn_\infty$ may be stronger than its external counterpart.)
\end{remark}

\begin{remark}
Throughout this section we could replace $\NN$ by an arbitrary infinite set $I$,
obtaining an axiom of ``$I$\=/indexed choice of dimension $\le d$'',
whose strength generally increases with the cardinality of $I$.
We could also impose this axiom for every set $I$,
meaning that arbitrary products of $(d-1)$\=/connected maps in $\cE$ are surjections.
Note that, even for a 1-topos $\cE$ and for $d = 0$,
this last axiom is weaker than the internal axiom of choice in $\cE$,
since it only concerns families of maps of $\cE$ indexed by an external set
(equivalently, internal families in $\cE$ indexed on a constant sheaf).

\cite{Roos:Rlim} introduced the axiom AB4*\=/$d$ for a Grothendieck abelian category $\cA$,
which states that for every set $I$,
the derived functors of the $I$\=/ary product functor $\prod_I : \cA^I \to \cA$
vanish in degrees higher than $d$.
When $d = 0$, this axiom reduces to Grothendieck's axiom AB4* (products are exact).
The axiom $\CC d$ is a non-abelian analogue of AB4*\=/$d$,
except that we consider only countable products.
\end{remark}

The following result will be a consequence of \cref{prop:enough-DCd} proved in the next section.
\begin{proposition}
\label{prop:bhd}
Any $n$\=/topos with enough objects of homotopy dimension $\leq d$ satisfies $\CC d$.
\end{proposition}

\begin{example}[The $d$\=/cube]
\label{ex:cube}
Let $\cE$ denote the \oo topos of sheaves on the $d$\=/cube $X = [0,1]^d$.
By classical dimension theory, every subset of $\RR^d$ has covering dimension $\le d$
\cite[Chapter 1]{Engelking:Dimension}.
In particular, every open subset $U \subseteq X$ has covering dimension $\le d$
so by \cite[Theorem 7.2.3.6]{Lurie:HTT}, $\cE\slice U$ has homotopy dimension $\le d$.
Then $\cE$ satisfies $\CC d$ by \cref{prop:bhd}.

We show that $\cE$ does not have $\CC {d-1}$.
We already checked this for $d = 1$ in \cref{ex:CCd:CCdnoCCd-1}, so assume $d \ge 2$.
We take $A \in \cE$ to be the Eilenberg--Mac\ Lane object $K(F, d)$
for a sheaf of abelian groups $F$ on $X$ constructed below
in such a way that $\pi_0(A^\NN)$ is not the terminal object of $\cE$.
Then $A$ will be $(d-1)$\=/connected, but $A^\NN$ will not be $0$\=/connected,
so $\cE$ does not satisfy $\CC {d-1}$.

We can compute the stalk of the sheaf $\pi_0(A^\NN)$ at a point $x \in X$ as
\begin{align*}
  \pi_0(A^\NN)_x
  & = \pi_0((A^\NN)_x)
  = \pi_0(\colim_{U \ni x} A^\NN(U))
  = \colim_{U \ni x} \pi_0(\prod_\NN A(U))
  = \colim_{U \ni x} \prod_\NN \pi_0(A(U)) \\
  & = \colim_{U \ni x} \prod_\NN H^d(U, F|_U).
\end{align*}
(Compare \cite[Proposition 1.6]{Roos:Rlim}.)
Here, $H^d$ denotes ordinary sheaf cohomology.
The first step used \cite[Remark 6.5.1.4]{Lurie:HTT},
and the last step used \cite[Remark 7.2.2.17]{Lurie:HTT}.
In this colimit, we may restrict attention to
the cofinal family of contractible open neighborhoods $U$ of $x$,
as $X$ is locally contractible.

Let $x = (\frac 12, \ldots, \frac 12)$ be the center of $X$
and let $Z \subseteq X$ be the (closed) union of $\{x\}$ and the spheres $S_k$ of radius $1/k$ centered at $x$
for all $k \ge 3$.
Let $j : X \setminus Z \to X$ denote the inclusion of the open complement of $Z$,
and set $F = j_! \ZZ$.
For $U$ a contractible open neighborhood of $x$,
the localization sequence
$0 \to (j_! j^* \ZZ)|_U \to \ZZ|_U \to (i_* i^* \ZZ)|_U \to 0$
yields an isomorphism $H^d(U, F|_U) \cong H^{d-1}(U \cap Z, \ZZ)$,
as $H^{d-1}(U, \ZZ) = H^d(U, \ZZ) = 0$.
Moreover, these isomorphisms are compatible with restriction to a smaller $U' \subseteq U$.

The space $Z$ has $S_k \cup \{x\}$ as a retract for each $k$,
so there is an element $\alpha \in \prod_{\NN} H^{d-1}(Z, \ZZ)$
such that, for each $k \ge 3$,
$\alpha_k$ restricts to a generator of $H^{d-1}(S_k, \ZZ)$.
Any contractible open neighborhood $U$ of $x$ contains $S_k$ for some $k \ge 3$,
and then the restriction of $\alpha$ to $\prod_\NN H^{d-1}(U \cap Z, \ZZ)$
is nonzero in the $k$th component.
Hence, $\alpha$ represents a nontrivial element of
$\colim_{U \ni x} \prod_\NN H^d(U, F|_U) = \pi_0(A^\NN)_x$.
\end{example}

\begin{proposition}
\label{prop:CC-for-hc}
If $\CC d$ holds in $\cE$, then it also holds in every $(n+1)$-topos $\cE\truncated n$ for $0\leq n\leq \infty$ (in particular, in the hypercompletion $\cE\truncated\infty$).
\end{proposition}

The proof of \cref{prop:CC-for-hc} will need a few lemmas about connected maps.
In this proof, we write $P_k = (-)\truncated k: \cE \to \cE\truncated k$
for the left adjoint to the inclusion $\cE\truncated k \subseteq \cE$
for any $0 \le k \le \infty$;
$P_\infty : \cE \to \cE\truncated \infty$
is the hypercompletion functor.

\begin{lemma}
\label{lem:cancel-conn}
If $f:X\to Y$ and $g:Y\to Z$ are two maps such that $g$ is $(k+1)$-connected and $gf$ is $k$-connected, then $f$ is $k$-connected.
\end{lemma}
\begin{proof}
We consider the pullback of $g$ along $gf$
\[
\begin{tikzcd}
X\ar[r] \ar[rd,equal] \ar[rr,bend left, "f"]& Y'\ar[r] \ar[d] \pbmark & Y \ar[d,"g"] \\
& X \ar[r,"gf"] & Z\,.
\end{tikzcd}
\]
The map $Y'\to Y$ is $k$-connected since it is a base change of $gf$.
The map $Y'\to X$ is $(k+1)$-connected since it is a base change of $g$.
The map $X\to Y'$ is $(k+1)$-connected since it is a section of a $k$-connected map (see \cite[Proposition 6.5.1.20]{Lurie:HTT} or \cite[Lemma 7.5.11]{hottbook}).
The map $f$ is the composition of two $k$-connected maps, thus $k$-connected.
\end{proof}

\begin{lemma}
\label{lem:Pn-equiv}
If a map $X\to Y$ is sent to an isomorphism by
$P_{k+1}:\cE\to \cE\truncated {k+1}$, then it is $k$-connected.
\end{lemma}
\begin{proof}
Apply \cref{lem:cancel-conn} to the commutative triangle $X\to Y \to P_{k+1}X=P_{k+1}Y$.
\end{proof}

\begin{lemma}
\label{lem:CC-for-trunc}
For $k+1 \le n \le \infty$, a map $X\to Y$ in $\cE$ is $k$-connected if and only if the map $P_nX\to P_nY$ is $k$-connected in $\cE\truncated n$.
(In other words, the functor $P_n:\cE\to \cE\truncated n$ preserves and reflects $k$\=/connected maps.
Note that we may have $k = n = \infty$.)
\end{lemma}
\begin{proof}
The factorization system ($k$-connected, $k$-truncated) is defined both in $\cE$ and in $\cE\truncated n$.
The inclusion $\cE\truncated n\subseteq \cE$ preserves diagonals and isomorphisms, thus $k$-truncated maps.
(Since it is moreover a conservative functor, it reflects $k$-truncated maps; we will use this below.)
Consequently, its left adjoint $P_n:\cE\to \cE\truncated n$ preserves $k$-connected maps for every $k$.

Conversely, we assume that the map $P_nX\to P_nY$ is $k$-connected in $\cE\truncated n$.
Let us show that it is also $k$-connected when viewed as a map in $\cE$.
We consider its ($k$-connected, $k$-truncated) factorization $P_nX\to Z\to P_nY$ in $\cE$.
The map $Z\to P_nY$ is $k$-truncated in $\cE$, thus $n$-truncated since $k<n$.
This shows that $Z$ is an object of $\cE\truncated n$.
Since the inclusion $\cE\truncated n\subseteq \cE$ reflects $k$-truncated maps, the map $Z\to P_nY$ is also $k$-truncated in $\cE\truncated n$.
Since $Z$ is in $\cE\truncated n$, the image of $P_nX\to Z$ by $P_n$ is itself.
Since it is $k$-connected in $\cE$, it is $k$-connected in $\cE\truncated n$ by the argument above.
Since $P_nX\to Z$ and $P_nX\to P_nY$ are $k$-connected in $\cE\truncated n$, so is  $Z\to P_nY$ by cancellation.
This shows that $Z\to P_nY$ is both $k$-connected and $k$-truncated in $\cE\truncated n$, thus an isomorphism.
Hence we have shown the map $P_nX\to P_nY$ is $k$-connected in $\cE$.

We now consider the canonical square in $\cE$
\[
\begin{tikzcd}
X \ar[r] \ar[d]& P_nX \ar[d]\\
Y \ar[r] & P_nY\,.
\end{tikzcd}
\]
The maps $Y\to P_nY$ is $n$-connected, thus $(k+1)$-connected.
The maps $X\to P_nX$ is $n$-connected, thus $k$-connected.
The map $P_nX\to P_nY$ is $k$-connected, so the composition $X\to P_nX\to P_nX$ is $k$-connected.
Then the map $X\to Y$ is $k$-connected by \cref{lem:cancel-conn}.
\end{proof}

\begin{remark}
For $n < \infty$, since not every map between $n$-connected objects is an $n$-connected map,
the functor $P_n:\cE\to \cE\truncated n$ inverts some maps which are not $n$-connected.
Since the $n$-connected maps in $\cE\truncated n$ are exactly the isomorphisms, this shows that \cref{lem:CC-for-trunc} is false when $k+1 > n$.
\end{remark}

\begin{remark}
For $n=0$, \cref{lem:CC-for-trunc} recovers that $X\to Y$ is a surjection if and only if the map $P_0X\to P_0Y$ is a surjection.
For $n=\infty$, the result is \cite[Proposition A.4.2.1 and Example A.4.2.4]{Lurie:SAG} or \cite[Proposition 4.3.2]{ABFJ:GT}.
\end{remark}

\begin{proof}[Proof of \cref{prop:CC-for-hc}]
If $n<\infty$, we have seen that $\CC d$ holds in the $(n+1)$-topos $\cE\truncated n$ for $d\geq n+1$.
So, we may assume that $d\le n$.

Let $(X_k\to Y_k)$ be a countable family of $(d-1)$\=/connected maps in $\cE\truncated n$,
and let $p : X\to Y$ be their product in $\cE\truncated n$.
We denote by $\iota_n$ the inclusion $\cE\truncated n \subseteq \cE$.
The map $\iota_n p : \iota_nX\to \iota_nY$ is the product of the maps $\iota_nX_k\to \iota_nY_k$ in $\cE$, since $\iota_n$ preserves products.
Using $P_n\iota_n=id$ and the reflection of $(d-1)$-connected maps of $P_n: \cE\to\cE\truncated n$ (\cref{lem:CC-for-trunc}), we deduce that each map $\iota_nX_k\to \iota_nY_k$ is $(d-1)$\=/connected in $\cE$.
By assumption on $\cE$, the map $\iota_n p:\iota_nX\to \iota_nY$ is surjective in $\cE$.
Thus the map $P_n\iota_n p = p$ is surjective in $\cE\truncated n$.
\end{proof}

A counterexample to the converse of \cref{prop:CC-for-hc} for $n=\infty$ will be given in \cref{rem:app-hc}.

\subsection{Dependent choice}

In this section, we introduce the stronger axiom of dependent choice.
We denote by $\Tower$ the poset $\dots\to 2\to 1\to 0$.
A diagram $X\index : \Tower\to \cE$ is called a \emph{tower}.
The \emph{composition} of the tower $X\index$ is the map $\lim X\index \to X_0$.

\begin{definition}[Dependent choice]
\label{def:DC}
Let $\cE$ be an $n$\=/topos (for $1\leq n\leq \infty$). 
For $-1\leq d\leq \infty$, we say that the axiom of \emph{dependent choice of dimension $\leq d$} ($\DC d$) holds in $\cE$ if the composition of a tower of $(d-1)$\=/connected maps is a surjection.
\end{definition}

The condition $\DC 0$ is called ``repleteness'' in \cite{Bhatt-Scholze:proetale,Mondal-Reinecke:Postnikov}.
An argument similar to that of \cref{lem:decalage} would show that $\DC d$ holds if and only if the limit of every tower of $(m+d)$\=/connected maps is an $m$\=/connected map for all $m \ge -2$.
By reindexing the tower, the axiom $\DC d$ holds if and only if, for every tower $X\index$ of $(d-1)$\=/connected maps, all the maps $\lim X\index \to X_k$ are surjections.

\begin{lemma}
\label{lem:DC-CC}
The axiom $\DC d$ implies $\CC d$.
\end{lemma}
\begin{proof}
For a countable family of $(d-1)$\=/connected maps $A_k\to B_k$, 
their product $A \to B$ is the limit of a tower of $(d-1)$\=/connected maps in $\cE\slice B$
\[
\dots
\to 
(A_0 \times A_1 \times A_2) \times_{B_0 \times B_1 \times B_2} B
\to 
(A_0\times A_1)\times_{B_0\times B_1}B
\to 
A_0\times_{B_0}B
\to B\,.\qedhere
\]
\end{proof}

\begin{examples}
\label{ex:DCd}
\begin{exmpenum}

\item\label{ex:DCd:strength} We always have $\DC {-1}\,\Ra\,\DC 0\,\Ra\,\DC 1\,\Ra\,\cdots\,\Ra\,\DC\infty$.

\item\label{ex:DC-1} The axiom $\DC {-1}$ says that the composition of an arbitrary tower is a surjection. This holds if and only if $\cE=1$.

\item\label{ex:DC0} The axiom $\DC 0$ says that the composition of a tower of surjections is a surjection. 
In a 1-topos, this is equivalent to the usual statement of dependent choice expressed in the internal language \cite[Proposition 3.35]{Mejak:Choice}.

Jensen proved \cite{Jensen:ZF} that there is a model of Zermelo--Fraenkel set theory that satisfies countable choice but not dependent choice \cite[Theorem 8.12]{Jech:AoC}.
An example of a 1-topos where $\CC 0$ holds but not $\DC 0$ is given in \cite[Section 4.3]{Mejak:Choice}.

\item\label{ex:DC-cube}
The $d$\=/cube $[0,1]^d$ of \cref{ex:cube} has enough objects of homotopy dimension $\leq d$ and satisfies $\DC d$ by \cref{prop:enough-DCd} below.
However, it does not satisfy $\DC {d-1}$, since we saw in \cref{ex:cube} that it does not satisfy $\CC {d-1}$.

\item\label{ex:DC-for-HC}
An argument similar to that of \cref{prop:CC-for-hc} would show that if $\DC d$ holds in an \oo topos $\cE$, then it holds in its hypercompletion $\cE\truncated\infty$.
The \oo topos of \cref{rem:app-hc} gives a counterexample to the converse.

\end{exmpenum}
\end{examples}

\begin{question}
For $d > 0$, does there exist an $n$\=/topos which satisfies $\CC d$, but not $\DC d$?
\end{question}

\begin{proposition}
\label{prop:enough-DCd}
An $n$\=/topos with enough objects of homotopy dimension $\leq d$ satisfies $\DC d$, hence also $\CC d$.
\end{proposition}
\begin{proof}
Let $\dots \to A_1\to A_0$ be a tower of $(d-1)$\=/connected maps.
We want to show that the composition $\lim A\index\to A_0$ is a surjection.
By assumption on $\cE$, there exists a covering family $(q_i: X_i\to A_0)$ where the $X_i$ are objects of homotopy dimension $\leq d$.
For each $i$, using the fact that the maps $A_{k+1}\to A_k$ are $(d-1)$\=/connected and \cref{lem:hdim-le-iff}, we can construct for every $k \ge 0$ in turn a lift $X_i\to A_{k+1}$ of the previous map $X_i\to A_k$, and therefore a map $X_i\to \lim A\index$ whose composition with $\lim A\index \to A_0$ is the original map $q_i$.
Since the map $\coprod X_i\to A_0$ is a surjection and admits a factorization $\coprod X_i\to \lim A\index\to A_0$, the map $\lim A\index\to A_0$ must be a surjection (see \cite[Corollary 6.2.3.12~(2)]{Lurie:HTT}).
The last assertion is \cref{lem:DC-CC}.
\end{proof}

\begin{example}
\label{ex:DC0:proj}
The pro-\'etale 1-topoi have enough objects of homotopy dimension $\leq 0$ and therefore satisfy $\DC 0$ \cite[Proposition 4.2.8]{Bhatt-Scholze:proetale}.
In particular, the 1-topos of condensed sets satisfies $\DC 0$.
In these examples, we assume that the sites of definition are cut off at an uncountable strong limit cardinal, as described in \cite[Remark 4.1.2]{Bhatt-Scholze:proetale}, so as to obtain a Grothendieck topos with enough projective objects.

The \emph{light} condensed sets \cite[Video 2/24]{Clausen-Scholze:analytic} form the 1-topos of sheaves on the category of light profinite sets, those with only countably many clopen subsets, for the topology generated by finite jointly surjective families. This 1-topos does not have enough objects of homotopy dimension $\leq 0$, but still satisfies $\DC 0$ \cite[Video 3/24, 49']{Clausen-Scholze:analytic}.
\end{example}

\section{Proofs of the main theorems}

This section will use the terminology and notations of \cref{app:Postnikov}, which the reader is advised to read first.

\medskip
\Cref{thm:main,thm:CC:1-to-infty} will follow from the following lemmas.
We fix an \oo topos $\cE$.
Recall that an $n$\=/category $K$ is of homotopy $\cE$\=/dimension $\leq d$ if it satisfies the conditions of \cref{lem:equiv-HD}.

\begin{lemma}
\label{lem:dim}
If $\NN$ is of homotopy $\cE$\=/dimension $\leq d$, then the poset $\Tower$ is of homotopy $\cE$\=/dimension $\leq d+1$.
\end{lemma}
\begin{proof}
By \cref{lem:equiv-HD}~\eqref{lem:equiv-HD:4} for $K=\NN$, $\NN$ is also of homotopy $\cE\slice E$\=/dimension $\leq d$ for every $E \in \cE$.
We will check \cref{lem:equiv-HD}~\eqref{lem:equiv-HD:2:obj} for $K=\Tower$.
We show that, in any slice $\cE\slice E$, if countable products of $(m+d+1)$\=/connected objects are $(m+1)$\=/connected, then $\Tower$\=/limits of $(m+d+1)$\=/connected objects are $m$\=/connected.

Recall that $\Tower$ is the free category on the graph $s,i:\NN\rightrightarrows \NN$, where $i$ is the identity and $s$ the successor map.
Then the limit of a diagram $\dots \to X_1\to X_0$ in $\cE\slice E$ is equivalent to the equalizer of the pair
$s,i:\prod X_k\rightrightarrows \prod X_k $, where $i$ is the identity map and $s$ is induced by the maps $X_{k+1}\to X_k$.
Equivalently, this equalizer is the pullback of the diagram
\[
\begin{tikzcd}
\lim X_k \ar[rr]\ar[d]\pbmarkk && \prod X_k \ar[d,"\Delta"]\\
\prod X_k \times \prod X_k \ar[rr,"{(s,i)}"] && \prod X_k \times \prod X_k\,.
\end{tikzcd}
\]
When each $X_k$ is an $(m+d+1)$\=/connected object, their product $\prod X_k$ is an $(m+1)$\=/connected object by assumption on $\NN$, and the diagonal of $\prod X_k$ is an $m$\=/connected map.
The map $\lim X_k\to \prod X_k \times \prod X_k$ is $m$\=/connected by base change,
so the map $\lim X_k\to 1$ is $m$\=/connected by composition.
\end{proof}

Recall that we denote the Postnikov tower of an object $X$ by $\dots \to X\truncated 1\to X\truncated 0$.

\begin{lemma}
\label{lem:hc=limpt}
If $\Tower$ has finite homotopy $\cE$\=/dimension, then, for every object $X$ in $\cE$, the map $X\to \lim X\truncated k$ is the hypercompletion of $X$.
\end{lemma}
\begin{proof}
Let $d$ be the homotopy $\cE$\=/dimension of $\Tower$.
We put $Y=\lim X\truncated k$.
For a fixed $n\geq d$, all the maps $X\to X\truncated {k+n}$ are $n$\=/connected.
The map $(X\to Y) = \lim \, ( X\to X\truncated {k+n} )$ is $(n-d)$\=/connected by assumption on $\Tower$.
Since $n$ was arbitrary, this shows the map is \oo connected.
Let $\cE\truncated\infty\subseteq \cE$ be the hypercompletion of $\cE$.
All truncated objects are in $\cE\truncated\infty$ and $\cE\truncated\infty$ is closed under limits, so $Y=\lim X\truncated k$ is in $\cE\truncated\infty$. 
This shows that the map $X\to Y$ is the hypercompletion of $X$.
\end{proof}

\begin{remark}
\label{rem:hc=limpt}
Putting together \cref{lem:dim,lem:hc=limpt}, we see that $\CC d$ implies that the hypercompletion of an object is given by the limit of its Postnikov tower.
This is a result interesting on its own.
It is equivalent to the convergence of Postnikov towers in $\cE\truncated\infty$, which is apparently weaker than the Postnikov completeness of $\cE\truncated\infty$ (i.e. weaker than \cref{thm:main}) though we do not know any counterexample (see \cref{app:Postnikov}).
The following lemma shows that $\CC d$ implies in fact a stronger result, of which \cref{thm:main} is a mere reformulation.
\end{remark}

\begin{lemma}
\label{lem:effectivity}
If $\Tower$ has finite homotopy $\cE$\=/dimension, then every Postnikov tower in $\cE$ is effective.
\end{lemma}
\begin{proof}
Suppose $\Tower$ has homotopy $\cE$\=/dimension $\le d$, with $d < \infty$.
Let $\dots \to X_1\to X_0$ be a formal Postnikov tower in $\cE$,
and put $X=\lim X_k$.
For a fixed $n\geq d$, all the maps $X_{k+n}\to X_{n}$ are $n$\=/connected.
The map $(X\to X_n) =\lim \, ( X_{k+n}\to X_n )$ is $(n-d)$\=/connected by assumption on $\Tower$.
In particular, it induces an equivalence $X\truncated{n-d} = (X_n)\truncated{n-d} = X_{n-d}$.
Since $n$ was arbitrary, this shows the tower is effective.
\end{proof}

\begin{remark}
\label{rem:HoTT}
Putting together \cref{lem:dim,lem:effectivity}, we deduce that ``if $\CC d$ holds for some finite $d$, then Postnikov towers are effective''.
This statement can be formulated and proven in homotopy type theory in essentially the same way.
\end{remark}

We can now prove the two results of the introduction.

\begin{proof}[Proof of \cref{thm:main}]
The Postnikov completion $(-)\tower : \cE \to \Post\cE$ and the hypercompletion $P_\infty : \cE \to \cE\truncated\infty$ invert the same class of maps, namely the \oo connected maps.
If all Postnikov towers are effective and so $(-)\tower:\cE \to \Post \cE$ is a localization, then the canonical morphism $\cE\truncated\infty\to \Post\cE$ is an equivalence.
When $\cE$ satisfies $\CC d$ for some finite $d$, this holds by \cref{lem:dim,lem:effectivity}.
\end{proof}

\begin{remark}
\label{rem:main:post}
By \cref{thm:main,prop:CC-for-hc}, if $\CC d$ holds in $\cE$, it holds also in  $\Post\cE$.
\end{remark}

\begin{remark}
\label{rem:main:infty}
Recall from \cref{ex:CCd:infty} that every hypercomplete topos satisfies $\CC\infty$.
Since there exist hypercomplete topoi which are not Postnikov complete, this shows that \cref{thm:main} is false for $d=\infty$.
\end{remark}

\begin{remark}
\label{rem:main:infty:2}
Recall from \cite[Corollary A.7.2.6]{Lurie:SAG} that the subcategory of Postnikov complete \oo topoi is closed under colimits of geometric morphisms.
The \oo topos of sheaves on the cube $[0,1]^d$ is Postnikov complete because it is locally of homotopy dimension $\leq d$ \cite[Proposition 7.2.1.10]{Lurie:HTT}.
The \oo topos $\coprod_d\,[0,1]^d$ of \cref{ex:CCinfty:noCCd} is Postnikov complete as a colimit of Postnikov complete \oo topoi.
This shows that the converse of \cref{thm:main} is false.
Another counterexample is given by the \oo topos of \cref{prop:noCC}.
\end{remark}

\begin{remark}
\label{rem:main:slice}
By \cref{ex:CCd:slice}, the conclusion of \cref{thm:main} holds for all slices of $\cE$.
We do not know if the slices of Postnikov complete topoi are always Postnikov complete.
\end{remark}

We can apply \cref{thm:main} to get a proof of \cite[Proposition 7.2.1.10]{Lurie:HTT}. 
\begin{corollary}
\label{cor:lhd-PC}
Any \oo topos locally of homotopy dimension $\leq d$ is Postnikov complete.
\end{corollary}
\begin{proof}
Let $\cE$ be such an \oo topos.
By \cref{lem:lochd-hc} we know that $\cE$ is hypercomplete.
By \cref{prop:bhd,thm:main}, $\cE$ is in fact Postnikov complete.
\end{proof}

\medskip

\begin{proof}[Proof of \cref{thm:CC:1-to-infty}]
Let $\cE_1$ be a 1-topos satisfying $\CC 0$;
we must show its \oo topos envelope $\cE$ also satisfies $\CC 0$.
Let $(A_k\to B_k)$ be a countable family of surjections in $\cE$.
We want to show that $p: \prod A_k\to \prod B_k$ is a surjection.
By construction of $\cE$, there exists a surjection $X\to \prod B_k$ by some object $X$ in $\cE_1$.
We put $C_k = X\times_{B_k}A_k$.
The map $C_k\to X$ is a surjection by base change.
Let $Y_k\to C_k$ be a cover by an object of $\cE_1$.
\[
\begin{tikzcd}
Y_k \ar[r, two heads] \ar[d, two heads] & C_k \ar[rr]\ar[d, two heads]  \pbmarkk && A_k\ar[d, two heads]\\
X \ar[r, equals] & X \ar[r, two heads,"u"] & \prod B_k \ar[r] & B_k
\end{tikzcd}
\qquad\qquad 
\begin{tikzcd}
\prod\slice X Y_k \ar[r]\ar[d, two heads] \pbmark &
\prod Y_k \ar[r]\ar[d, two heads]&\prod A_k \ar[d, "p"]\\
X \ar[r, "\Delta"]\ar[rr, two heads,bend right,"u"] &X^\NN \ar[r, two heads] & \prod B_k
\end{tikzcd}
\]
The composite map $Y_k\to C_k \to X$ is a surjection in $\cE$ between objects of $\cE_1$.
It is also a surjection in $\cE_1$ by \cref{lem:reflect-surj}.
The object $\prod\slice X Y_k$ is in $\cE_1$ since discrete objects are closed under limits.
By the hypothesis on $\cE_1$, the product map $\prod Y_k\to X^\NN$ is a surjection in $\cE_1$, and so is the map $\prod\slice X Y_k\to X$ by base change along $X\to X^\NN$.
It is also a surjection in $\cE$, since the embedding $\cE_1\to\cE$ preserves surjections.
The map $\prod\slice X Y_k\to X\to \prod B_k$ is a surjection as a composition of surjections.
Since it factors through $p : \prod A_k \to \prod B_k$, $p$ is also a surjection.
\end{proof}

\begin{remark}
\label{rem:DC:1-to-infty}
Using a similar argument, it is possible to show that if $\cE_1$ is a 1-topos where $\DC 0$ holds, then $\DC 0$ also holds in its \oo topos envelope $\cE$.
\end{remark}

\begin{question}
\label{rem:CC:d-to-infty}
Let $\cE$ be the \oo topos envelope of a $(d+1)$\=/topos $\cE_d$ satisfying $\CC d$ (or $\DC d$) for $d > 0$.
Does $\cE$ satisfy $\CC d$ (or $\DC d$)?
\end{question}

By putting together \cref{thm:main,thm:CC:1-to-infty} with \cref{ex:DC0}, we get a proof of \cite[Theorem A]{Mondal-Reinecke:Postnikov}.

\begin{corollary}
\label{cor:mondal-reinecke}
If a 1-topos $\cE_1$ satisfies $\CC 0$ (a fortiori $\DC 0$), then the hypercompletion of the \oo topos envelope of $\cE_1$ is Postnikov complete.
\end{corollary}

\section{A topos without \texorpdfstring{$\CC\infty$}{countable choice of dimension infinity}}
\label{sect:noCC}

This section gives an example of an \oo topos where the axiom $\CC d$ fails for every $d \le \infty$ (\cref{prop:noCC}).
As a technical step, we start by giving an example of an \oo connected object with no global section.
We shall use the object introduced in \cite[Example A.10]{DHI}, but following the exposition of \cite[11.3]{Rezk:topos}.
We denote by $\cR$ the \oo topos constructed there and by $R$ the distinguished \oo connected object.
Rezk proves that the space of global sections of $R$ is (noncanonically) equivalent to the zeroth space of the sphere spectrum, $\Gamma(R) \simeq QS^0$.
Let $a,b:1\to R$ be two global sections in the classes of 0 and 1 in $\pi_0(\Gamma(R))=\ZZ$. 
We denote by $\Omega_{a,b}R$ the corresponding path object (fiber product) in $\cR$.
Because $R$ is \oo connected, so is $\Omega_{a,b}R$, but its space of global sections is now empty:
\[
\Gamma(\Omega_{a,b}R)
\ =\ \Omega_{0,1}\Gamma(R)
\ =\ \Omega_{0,1}QS^0
\ =\ \emptyset\,.
\]

More generally, such \oo connected objects can be build from any non-trivial \oo connected object.
Let $\cE$ be an \oo topos with a non-trivial \oo connected object $E$.
There exists $X \in \cE$ such that $\Map(X, E)$ is not contractible.
By replacing $\cE$ by $\cE\slice X$ and $E$ by $E\times X\to X$, we may assume $X = 1$.
If the space of global sections $\Gamma(E)$ is already empty,
then $E$ is the sought object.
Otherwise, choose a global section to serve as basepoint of $E$.
There exists an $0\leq n<\infty$ such that $\Gamma(E)$ is $(n-1)$\=/connected but not $n$\=/connected.
In other words, there exist two pointed maps $a,b:S^n\to E$
which lie in distinct components of $\Gamma(\Omega^n E)$.
The object we are looking for is the path space $\Omega_{a,b}(\Omega^n E)$.
(This is really \cref{lem:hdim-le-iff} \eqref{lem:hdim-le-iff:section} $\Ra$ \eqref{lem:hdim-le-iff:section-m} with $m = d = \infty$, $X = 1$.)

\medskip
Let $\Fin\subseteq \cS$ be the full subcategory of finite spaces (the spaces that can be obtained from the point by finite colimits).
The algebraically free \oo topos on one generator is $\S X = \Fun\Fin\cS$ and the universal object $X$ is the canonical inclusion $\Fin \to \cS$ \cite[Definition 2.2.6]{ABFJ:GT}.
For any \oo topos $\cE$, the category of algebraic morphisms $\S X \to \cE$ is equivalent to $\cE$.
There exists a topological localization $\S X \to \S {X\infconn}$ forcing the universal object $X$ to become \oo connected \cite[Example 3.1.3~(c) and Example 4.1.6~(d)]{ABFJ:GT}.
This means that for any \oo topos $\cE$, the category of algebraic morphisms $\S {X\infconn} \to \cE$ is equivalent to the full subcategory of $\cE$ spanned by the \oo connected objects.

\begin{lemma}
\label{lem:noGamma}
The universal \oo connected object $X\infconn:\Fin \to \cS$ satisfies $X\infconn(\emptyset)=\emptyset$.
\end{lemma}
\begin{proof}
The value $X\infconn(\emptyset)$ is the space of global sections, since $\emptyset$ is terminal as an object of $\Fin\op$.
If this space was non-empty, every \oo connected object of every \oo topos would have at least one global section.
Since we have seen that there exist \oo connected objects with no global sections, we must have $X\infconn(\emptyset)=\emptyset$.
\end{proof}

Let $\Fin^{(\NN)}\subseteq \Fin^\NN$ be the full subcategory of sequences of finite spaces whose values at all but finitely many indices are the empty space.
The algebraically free \oo topos $\S {X_0,X_1,\dots}$ on a countable set of generators is $\Fun{\Fin^{(\NN)}}\cS$.
We consider its localization $\cE=\S {X_0\infconn,X_1\infconn,\dots}$ forcing all the $X_k$ to become \oo connected.
The object $X_k\infconn$ is the functor $X\infconn\circ p_k:\Fin^{(\NN)}\to \Fin\to \cS$
where $p_k$ is the projection to the $k$th component.

\begin{proposition}
\label{prop:noCC}
The axiom $\CC \infty$ (and thus any $\CC d$ or $\DC d$) is false in $\cE$.\end{proposition}
\begin{proof}
In fact, we show that the countable product $\prod  X_k\infconn$ is the initial object of $\cE$.
This product can be computed in the category of functors $\Fin^{(\NN)}\to \cS$.
Let $A\index$ be an object in $\Fin^{(\NN)}$.
There exists some index $n$ such that $A_n=\emptyset$, so $(\prod  X_k\infconn) (A\index) = \prod  X\infconn(A_k) = \emptyset$.
Since $A\index$ was arbitrary, this shows that the functor $\prod  X_k\infconn:\Fin^{(\NN)}\to \cS$ is the initial functor, 
so it must also be the initial object of $\cE$.
\end{proof}

\begin{remark}
\label{rem:app-hc}
The hypercompletion of $\S {X\infconn}$ is $\cS$:
the canonical algebraic morphism $\S {X\infconn}\to \cS$ sending ${X\infconn}$ to 1 is a cotopological localization since it is generated by inverting the \oo connected map ${X\infconn}\to 1$, and because $\cS$ is hypercomplete, it must be the maximal cotopological localization, that is, the hypercompletion of $\S {X\infconn}$.
A similar argument shows that the hypercompletion of $\cE$ is also $\cS$.
Since $\cS$ is Postnikov complete, this shows that the converse of \cref{thm:main} is false.
This also provides a counterexample to the converse of \cref{prop:CC-for-hc}.
\end{remark}

\appendix
\section{Postnikov separation, convergence, effectivity, and completeness}
\label{app:Postnikov}

This appendix summarizes the various conditions of convergence on Postnikov towers and introduces some terminology.
The original material is in \cite[Section 5.5.6]{Lurie:HTT}.

\medskip
The conventions on truncation are at the beginning of the paper.
We fix a presentable category $\cC$.
The \emph{Postnikov completion} of $\cC$ is the limit $\Post\cC$ of the tower $\dots \to \cC\truncated 1\to \cC\truncated 0$, computed in the category of categories, with transition functors given by truncation.
The category $\Post\cC$ is presentable.
Its objects are towers $X\index:\Tower\to\cC$ such that the map $X_{k+1}\to X_k$ identifies $X_k$ with the $k$\=/truncation of $X_{k+1}$.
We shall call such a tower a \emph{formal Postnikov tower} (\cite{Lurie:HTT} calls it a Postnikov pretower).

The \emph{Postnikov tower} $X\tower$ of an object $X$ is defined as the tower of truncations $\dots \to X\truncated 1\to X\truncated 0$.
This defines a functor $(-)\tower:\cC\to\Post\cC$ whose right adjoint sends a tower to its limit.
For any object $X$ in $\cC$, the unit of the adjunction $(-)\tower:\cC\rightleftarrows \Post\cC:\lim$ is the canonical map $X\to \lim_n X\truncated n$,
while for any formal Postnikov tower $X\index$ with limit $X$, the counit is the canonical morphism of towers $X\tower\to X\index$.

A map in $\cC$ is a \emph{Postnikov equivalence} if it is inverted by $\cC\to\Post\cC$.
We shall say that an object is \emph{\oo truncated} if it is local for Postnikov equivalences.
The \emph{Postnikov separation} of $\cC$ is defined as the localization $\cC\to \cC\hred$ with respect to Postnikov equivalences.
The category $\cC\hred$ is equivalent to the full subcategory spanned by \oo truncated objects.
This gives a factorization
\[
\begin{tikzcd}[sep = small]
\cC \ar[r]
& \cC\hred \ar[r]
& \Post\cC = \lim\, \cC\truncated n\,,
\end{tikzcd}
\]
where the first functor is a localization and the second functor is conservative. 
The truncations of an object $X$ and their limit fit into a diagram
\[
\begin{tikzcd}[sep=small]
X   \ar[r]
&X\truncated\infty
    \ar[r]
&\lim_n X\truncated n \ar[r]
&\dots\ar[r] 
& X\truncated 1
    \ar[r]
& X\truncated 0 
    \ar[r]
& X\truncated {-1}
    \ar[r]
&X\truncated {-2}=1\,.
\end{tikzcd}
\]

\begin{remark}
The factorization $\cC\to \cC\truncated\infty \to \Post\cC$ is typical of separation--completion processes. 
Other examples includes the separation--completion $A \to A/\bigcap I^n \to \lim A/I^n$ of a ring $A$ at an ideal $I$, 
or the separation--completion $X \to X/(d=0) \to \overline X$ of a pseudometric space $(X,d)$.
For this reason, we suggest replacing the terminology \emph{hypercompletion} by \emph{Postnikov separation}.
A similar terminology is used for prestable categories in \cite[Definition C.1.2.12]{Lurie:SAG}.
\end{remark}

\begin{remark}
When $\cC=\cE$ is an \oo topos,
for every $-2\leq n\leq\infty$, the localization $\cE\to\cE\truncated n$ is produced by inverting all $n$\=/connected maps.
In particular, the Postnikov separation $\cE\hred$ coincides with the hypercompletion of $\cE$ and is an \oo topos.
By \cite[Theorem A.7.2.4 and Proposition A.7.3.4]{Lurie:SAG}, the category $\Post \cE$ is an \oo topos and the functor $(-)\tower:\cE\to \Post\cE$ is an algebraic morphism of \oo topoi.
Moreover, the two functors of the factorization $\cE\to \cE\hred\to \Post\cE$ are algebraic morphisms of \oo topoi.
\end{remark}

\begin{definition}
\label{app:def:Postnikov-tower}
\begin{enumerate}
\item We say an object $X \in \cC$ \emph{has a convergent Postnikov tower}
if the unit map $X\to \lim_n X \truncated n$ is invertible,
i.e., if $X$ is the limit of its Postnikov tower.

\item We say a formal Postnikov tower $X\index$ is \emph{effective}
if the counit map $(\lim_n X_n)\tower\to X\index$ is invertible, i.e., if $X\index$ coincides with the Postnikov tower of its limit. 
\end{enumerate}
\end{definition}

\begin{definition}
\label{app:def:Postnikov}
A presentable category $\cC$ is said to be
\begin{enumerate}
\item \emph{Postnikov separated} if $\cC\to \Post\cC$ is conservative;
\item \emph{Postnikov convergent} if $\cC\to \Post\cC$ is fully faithful;
\item \emph{Postnikov effective} if $\cC\to \Post\cC$ is a localization (i.e. if  $\lim:\Post\cC\to \cC$ is fully faithful);
\item \emph{Postnikov complete} if $\cC\to \Post\cC$ is an equivalence.
\end{enumerate}
\end{definition}

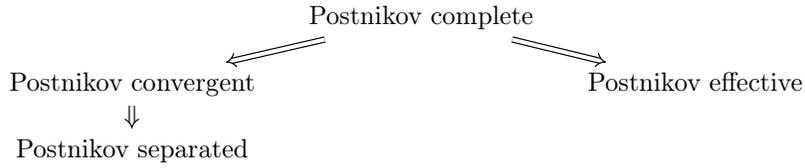
\[
\begin{tikzcd}[sep=small]
&\text{Postnikov complete}\ar[rd,Rightarrow]\ar[ld,Rightarrow]\\
\text{Postnikov convergent}\ar[d,Rightarrow] && \text{Postnikov effective}\\
\text{Postnikov separated}
\end{tikzcd}
\]

\begin{lemma}
\label{app:lem:Postnikov}
A presentable category $\cC$ is:
\begin{enumerate}

\item\label{app:lem:Postnikov:sep} Postnikov separated
if and only if it is hypercomplete
($X=X\truncated\infty$, for every object $X$); 

\item\label{app:lem:Postnikov:conv} Postnikov convergent 
if and only if every object has a convergent Postnikov tower
($X=\lim X\truncated n$, for every object $X$); 

\item\label{app:lem:Postnikov:eff} Postnikov effective 
if and only if every formal Postnikov tower is effective,
if and only if $\cC\hred$ is Postnikov effective,
if and only if $\cC\hred$ is Postnikov complete,
if and only if $\cC\hred=\Post\cC$;

\item\label{app:lem:Postnikov:comp} Postnikov complete 
if and only if it is Postnikov convergent and Postnikov effective,
if and only if it is Postnikov separated and Postnikov effective.
\end{enumerate}
\end{lemma}
\begin{proof}
Most of the statements are reformulations of \cref{app:def:Postnikov-tower,app:def:Postnikov}.
We only give an argument for \eqref{app:lem:Postnikov:comp}.
If formal Postnikov towers are effective, the functor $\cC\to\Post\cC$ is a localization and is therefore an equivalence if and only if it is fully faithful or conservative.
\end{proof}

\begin{remark}
The first equivalence of \cref{app:lem:Postnikov}~\eqref{app:lem:Postnikov:comp} is Lurie's criteria for Postnikov completeness in \cite[Proposition 5.5.6.26]{Lurie:HTT}.
With the notations therein, the implication $(1)\Ra(2)$ is what we call ``Postnikov convergent'', and the implication $(2)\Ra(1)$ is what we call ``Postnikov effective''. 
\end{remark}

\begin{remark}
\Cref{app:lem:Postnikov}~\eqref{app:lem:Postnikov:eff} implies in particular that if a category $\cC$ is Postnikov effective, the hypercompletion of any object $X$ is the limit of its Postnikov tower: $X\truncated\infty=\lim X\tower$.
\end{remark}

\begin{remark}
\label{rem:example}
An example of an \oo topos which is hypercomplete (Postnikov separated) but not Postnikov convergent is detailed in \cite[Example 3.17]{Mondal-Reinecke:Postnikov}.
By \cref{app:lem:Postnikov}~\eqref{app:lem:Postnikov:comp}, this is also an example of an \oo topos which is not Postnikov effective.
An example of an \oo topos which is Postnikov effective but not hypercomplete is the \oo topos $\cS[X\infconn]$ classifying \oo connected objects (see \cref{rem:app-hc}).
We do not know an example of an \oo topos or even a presentable category which is Postnikov convergent but not Postnikov effective.
\end{remark}

\providecommand{\bysame}{\leavevmode\hbox to3em{\hrulefill}\thinspace}

\end{document}